\documentclass[a4paper,12pt]{amsart}
\usepackage{ amsmath, tabularx, amsthm, amssymb, amsfonts}
\usepackage[inner=2.3cm,outer=2.3cm, bottom=3.3cm]{geometry}
\usepackage{setspace}
\usepackage[colorlinks=true,linkcolor=blue,urlcolor=blue,pagebackref]{hyperref}
\usepackage{color}
\usepackage[all]{xy}  
\usepackage{comment}
\usepackage{enumerate}
\usepackage{mathrsfs}              
\usepackage{stmaryrd}
\usepackage{colonequals}

\theoremstyle{definition}
\newtheorem{theorem}{Theorem}[section]

\newtheorem{question}[theorem]{Question}
\newtheorem{corollary}[theorem]{Corollary}
\newtheorem{lemma}[theorem]{Lemma}
\newtheorem{proposition}[theorem]{Proposition}

\newtheorem{notation}[theorem]{Notation}

\theoremstyle{definition}
\newtheorem{definition}[theorem]{Definition}
\newtheorem{setting}[theorem]{Setting}
\newtheorem{example}[theorem]{Example}

\newtheorem{remark}[theorem]{Remark}
\numberwithin{equation}{subsection}

\newtheorem{theoremx}{Theorem}

\newcommand{\m}{\mathfrak{m}}

\newcommand{\NN}{\mathbb{N}}
\newcommand{\ZZ}{\mathbb{Z}}

\newcommand{\Spec}{\operatorname{Spec}}

\newcommand{\Hom}{\operatorname{Hom}}

\renewcommand{\a}{\mathfrak{a}}

\newcommand{\q}{\mathfrak{q}}

\newcommand{\ov}[1]{\overline{#1}}

\newcommand{\n}{\mathfrak{n}}

\newcommand{\lr}[1]{{\langle {#1} \rangle}}

\newcommand{\difM}[1]{^{\lr{#1}_{\mathrm{mix}}}}

\renewcommand{\geq}{\geqslant}
\renewcommand{\leq}{\leqslant}

\newcommand{\Cp}{\mathcal{C}_p}

\newcommand{\Dq}[1]{^{\{{#1}\}}}

\author[De Stefani]{Alessandro De Stefani}
\address{Dipartimento di Matematica, Universit{\`a} di Genova, Via Dodecaneso 35, 16146 Genova, Italy}
\email{destefani@dima.unige.it}

\author[Grifo]{Elo\'isa Grifo}
\address{Department of Mathematics, University of Nebraska, Lincoln, NE 68588-0130, USA}
\email{grifo@unl.edu}

\author[Jeffries]{Jack Jeffries}
\address{Department of Mathematics, University of Nebraska, Lincoln, NE 68588-0130, USA}
\email{jack.jeffries@unl.edu}

\begin{document}
\newcommand{\tens}{\otimes}
\newcommand{\hhtest}[1]{\tau ( #1 )}
\renewcommand{\hom}[3]{\operatorname{Hom}_{#1} ( #2, #3 )}
\renewcommand{\t}{\operatorname{type}}

\title[A uniform Chevalley Theorem in mixed characteristic]{A uniform Chevalley Theorem for direct summands of polynomial rings in mixed characteristic}

\maketitle

\begin{abstract}
We prove an explicit uniform Chevalley theorem for direct summands of graded polynomial rings in mixed characteristic. Our strategy relies on the introduction of a new type of differential powers that does not require the existence of a $p$-derivation on the direct summand.
\end{abstract}

\section{Introduction}

A classical result of Chevalley \cite[II, Lemma 7]{Chevalley} asserts that, if $(R,\m)$ is a complete local ring and $\{I_n\}$ is a decreasing family of ideals of $R$ such that $\bigcap_{n \geq 0} I_n= (0)$, then there exists a function $f:\NN \to \NN$ such that $I_{f(n)} \subseteq \m^n$ for all $n \in \NN$. Put in different words, Chevalley's Theorem states that if $\bigcap_{n \geq 0} I_n=(0)$ then the topology induced by $\{I_n\}$ is finer than the $\m$-adic topology. In the special case of symbolic powers, that is, when $I_n = I^{(n)}$ for some proper ideal $I$ in $R$ such that $\bigcap_{n \geq 1} I^{(n)}=(0)$, Huneke, Katz, and Validashti show that if $(R,\m)$ is also reduced, then one can choose $f(n) = nC$ for some integer $C>0$ that is independent of $I$ \cite[Theorem 2.3]{HKV09}. More generally, if $R$ is Noetherian, $\m$ is a maximal ideal, and $R_\m$ is analytically unramified, there exists $C>0$ such that $I^{(nC)} \subseteq \m^n$ for all $n \geq 1$ and all ideals $I \subseteq \m$ such that $\{I^{(n)}\}$ defines topology that is finer than the $\m$-adic topology \cite[Corollary~2.4]{HKV09} (see also \cite{Swanson} for an analogous result for the comparison between ordinary and symbolic powers of $I$). 

Finding an explicit $C$ such that $I^{(nC)} \subseteq \m^n$ gives a lower bound, $n$, on the $\m$-adic order of elements in $I^{(nC)}$. We can think of this type of statement as a local ring version of computing minimal degrees $\alpha(-)$ for symbolic powers of homogeneous ideals in a graded ring.

This article deals with the case when $R=A[f_1,\ldots,f_t]$ is a graded direct summand of a polynomial ring $S=A[x_1,\ldots,x_m]$, with $\deg(x_i)>0$ and where $(A,\m_A)$ is, for the moment, any regular local ring. Notable examples of such rings are Veronese subrings and Segre products of polynomial algebras over $A$. Rings of invariants of $S$ under the action of a finite group $G$ are also of this type, provided the order of $G$ is invertible in $A$. In this setup, if we let $\m=(f_1,\ldots,f_t)R + \m_AR$, then there exists a uniform $C>0$ such that $Q^{(nC)} \subseteq \m^n$ for all homogeneous prime ideals $Q$ of $R$. In fact, if $\widehat{R}$ denotes the completion of $R$ at $\m$, we have that $\widehat{R}$ is a complete normal local domain. It then follows from \cite[Proposition 2.4]{HKV} (see also \cite{HKVERRATA}) that, for every prime $Q$ of $R$ that is contained in $\m$, there exists $C>0$ such that $Q^{(nC)} \subseteq (Q\widehat{R})^{(nC)} \cap R \subseteq (Q\widehat{R})^n \cap R = Q^n \subseteq \m^n$ for every $n\geq 1$. The aforementioned result of Huneke, Katz, and Validashti guarantees then that there is a uniform $D>0$ such that $Q^{(nD)} \subseteq \m^n$ for every prime ideal $Q \subseteq \m$ and every $n \geq 1$. However, finding an explicit value for $D$ generally proves to be an extremely challenging task.

When $A$ is a perfect field, an explicit answer is provided in \cite[Theorem 3.7]{SurveySP}: in the notation introduced above, if $D = \max\{\deg(f_1),\ldots, \deg(f_t)\}$, then $Q^{(nD)} \subseteq \m^n$ for every homogeneous prime ideal $Q$ of $R$ and every positive integer $n$. The main goal of this article is to extend this result to the mixed characteristic setting.

\begin{theoremx}[Corollary \ref{main Coroll}] \label{THMX A}
Let $p$ be a prime integer, and $A$ be a discrete valuation ring with uniformizer $p$ and perfect residue field. Let $S=A[x_1,\ldots,x_m]$, with $\deg(x_i) >0$ for all $i$, and $f_1,\ldots,f_t \in S$ be homogeneous elements of degree at most $D$ such that $R=A[f_1,\ldots,f_t]$ is a graded direct summand of $S$. If we let $\m=(f_1,\ldots,f_t,p)R$, then $Q^{(nD)} \subseteq \m^n$ for all $n \geq 1$ and all homogeneous prime ideals $Q$ in $R$.
\end{theoremx}

There are examples where $D$ as in Theorem \ref{THMX A} is optimal; in this sense, our result is sharp (see Example \ref{Example sharp}). For primes $Q$ that contain $p$, our strategy only requires that $R/(p)$ is a direct summand of $S/(p)$ (see Theorem \ref{main THM} and Remark \ref{Remark direct summand mod p}), and it actually works even if $Q$ is not homogeneous, as long as it is contained in $\m$. On the other hand, for homogeneous primes $Q \subseteq R$ that do not contain $p$ we actually get the stronger statement that $Q^{(nD)} \subseteq (f_1,\ldots,f_t)^n$ (see Theorem \ref{THM no p}). We also obtain analogous results for $A=\ZZ$. 

Our proofs makes crucial use of $p$-derivations, a notion introduced independently by Joyal \cite{Joyal} and Buium \cite{Buium1995}, and related to symbolic powers in mixed characteristic \cite{ZNmixed}. The main challenge is that, while $S$ as above always has a $p$-derivation \cite[Proposition 2.7]{ZNmixed}, $R$ may in principle lack one. This is circumvented by defining a new type of differential power that uses the $p$-derivation on $S$ combined with the fact that $R$ is a graded direct summand of $S$ (see Definition \ref{Defn Dq}). 

Using several improvements of the Noether bounds, we obtain the following application of our main result.

\begin{theoremx}(see Theorem \ref{THM Noether bound})
Let $p$ be a prime integer, and $A$ be a discrete valuation ring with uniformizer $p$ and perfect residue field. Let $G$ be a finite group of order $D$ coprime with $p$ acting on $S=A[x_1,\ldots,x_n]$, where $\deg(x_i)=1$ for all $i$. Let $R=S^{G}$ be the ring of invariants, and $\m = R_{>0} + pR$. Then $Q^{(nD)} \subseteq \m^n$ for every homogeneous prime ideal $Q \subseteq R$ and every integer $n$.
\end{theoremx}

In Section \ref{Section comparison}, we compare the newly introduced notion of differential powers with the ones already available in the literature. We also present some concrete examples in which our results can be applied.

\section{Differential powers of direct summands}

Let $R$ be a commutative ring with multiplicative identity. Given an ideal $\a \subseteq R$, we let $W$ be the complement of the union of the minimal primes of $\a$. The $n$th symbolic power of $I$ is defined as $\a^{(n)} = \a^nR_W \cap R$. In particular, when $Q$ is a prime, the $n$th symbolic power $Q^{(n)} = Q^nR_Q \cap R$ is just the $Q$-primary component of the ordinary power $Q^n$.

A classical result due to Zariski and Nagata \cite{Zariski,Nagata} identifies the $n$th symbolic power of a prime ideal $Q \subseteq\mathbb{C}[x_1,\ldots,x_m]$ as the ideal of all functions that vanish up to order~$n$ along the variety defined by $Q$. This was extended in several ways, first to include the case of polynomial rings over a perfect field \cite{SurveySP}, and then to cover the case of certain polynomial rings of mixed characteristic \cite{ZNmixed}. As the latter will be particularly relevant for the purposes of this article, we now recall the main notions and the results obtained by these three authors in \cite{ZNmixed}.

\begin{definition} 
Let $A$ be a commutative ring with unity, and $R$ be an $A$-algebra. The $R$-module of $A$-linear differential operators on $R$ of order at most $n$ is the $R$-module $D^n_{R|A}$ defined inductively as follows:
\begin{itemize}
\item $D^0_{R|A} = \Hom_R(R,R) \subseteq \Hom_A(R,R)$.
\item $D^n_{R|A} = \{\partial \in \Hom_A(R,R) \mid [\partial,r] \in D^{n-1}_{R|A}$ for every $r \in R\}$.
\end{itemize}
\end{definition}

If $R$ is an $A$-algebra, and $I \subseteq R$ is an ideal, we define the $n$th ($A$-linear) differential power of $I$ as
\[
I^{\langle n \rangle_A} \colonequals \{f \in R \mid \partial(f) \in I \text{ for all } \partial \in D^{n-1}_{R|A}\}.
\]

It can be shown that $I^{\langle n \rangle_A}$ is an ideal of $R$ \cite[Proposition 2.4]{SurveySP}. Moreover, if $\a$ is a $Q$-primary ideal, then so is $\a^{\langle n \rangle_A}$ for every $n \geq 1$; in particular, $\a^{(n)} \subseteq \a^{\langle n \rangle_A}$ (see \cite[Proposition 3.2]{ZNmixed}).

\begin{definition}\cite{Joyal,Buium1995} Let $p$ be a prime integer, and $S$ be a ring over which $p$ is a nonzerodivisor. A set-theoretic map $\delta\!: S \to S$ is called a $p$-derivation if the map $\phi_p:S \to S$ defined as $\phi_p(x) = x^p+p\delta(x)$ is a ring homomorphism.
\end{definition}

Equivalently, one can check that $\delta$ is a $p$-derivation if $\delta(1) = 0$ and for all $x,y \in S$ we have
\[
\delta(xy) = x^p\delta(y) + \delta(x)y^p + p\delta(x)\delta(y)
\]
and
\[
\delta(x+y) = \delta(x) + \delta(y) + \mathcal{C}_p(x,y),
\]
where $\mathcal{C}_p(x,y) = \frac{x^p+y^p-(x+y)^p}{p}$. The map $\phi_p$ is called a lift of Frobenius, since it induces the Frobenius map on $S/(p)$. 

\

We will need some properties of $p$-derivations, which we include in the following technical lemma.

\begin{lemma} \label{Lemma Prop delta} 
Fix a prime $p \in \mathbb{Z}$, let $S$ be a ring in which $p$ is a nonzerodivisor, and let $\delta$ be a $p$-derivation on $S$. Given $x,y \in S$, and any $n \geq 1$, we have 
\begin{enumerate}
\item $\delta^n(x+y) -\delta^n(x) \in (y,\delta(y),\ldots,\delta^n(y))S$.
\item $\delta^n(F(x)) \in (x,\delta(x),\ldots,\delta^{n}(x),p\delta^{n+1}(x))S$.
\item $\delta^n (xy) - x^{p^n}\delta^n(y) \in (y,\delta(y),\ldots,\delta^{n-1}(y),p\delta^n(y))S$.

\noindent In particular,
\begin{enumerate}[(a)]
\item $\delta^n(xy) \in (y, \delta(y), \delta^2(y), \ldots, \delta^n(y))$, and 
\item $\delta^n(py) \in (y, \delta(y), \delta^2(y), \ldots, \delta^{n-1}(y), p\delta^n(y))$.
\end{enumerate}
\end{enumerate}
\end{lemma}

\begin{proof} 
First, note that $\Cp(x,y) \in xS$, a fact that we will use repeatedly. Note also that (a) and (b) are immediate consequences of (3), and thus do not require proof. 

We prove all claims at once using induction on $n \geq 1$. First, take $n=1$. We have
\[
\delta(x+y) = \delta(x)+ \delta(y) + \Cp(x,y) =\delta(x) + z_1,
\]
where $z_1 = \delta(y) + \Cp(x,y) \in (y,\delta(y))S$ because $\Cp(x,y) \in yS$. This proves the base case of (1). For (2), observe that 
\[
\delta(F(x)) = \delta(x^p+p\delta(x)) = \delta(x^p)  + \delta(p\delta(x)) + \Cp(x^p,p\delta(x)).
\]
On the one hand, $\Cp(x^p, p\delta(x)) \in xS$. Moreover, $\delta(p\delta(x)) \in (\delta(x), p \delta^2(x))$, since
$$\delta(p\delta(x)) = p^{p-1} \delta^2(x) + \delta(x)^p \delta(p) + p \delta^2(x) \delta(p).$$
Finally, since $\delta(x^p)  = x^p\delta(x^{p-1}) + x^{p^2-p}\delta(x) + p\delta(x^{p-1})\delta(x)$, we conclude that $\delta(F(x)) \in (x,\delta(x),p\delta^2(x))S$, as claimed.

For (3), note that
\[
\delta(xy) = x^p\delta(y) + F(y)\delta(x)  = x^p\delta(y) + z_2,
\]
where $z_2 \in F(y)S \subseteq (y,p\delta(y))S$.

We now assume all claims hold for $n-1$, and prove that they also hold for $n$, starting from (1). Observe that 
\[
\delta^n(x+y) = \delta^{n-1}(\delta(x+y)) = \delta^{n-1}(\delta(x)+z_1')
\]
where $z_1' \in (y,\delta(y))S$. By induction, we know that $\delta^{n-1}(\delta(x)+z_1') = \delta^n(x) + z_1$ for some $z_1 \in (z_1',\delta(z_1'),\ldots,\delta^{n-1}(z_1'))S$. As $z_1' \in (y,\delta(y))S$, we can write it as $a\delta(y) + by$ for some $a,b\in S$. Our inductive hypothesis on (1) guarantees that for all $1 \leq j \leq n-1$ we have $\delta^j(z_1') = \delta^j(a \delta(y)) + z_3$, where $z_3 \in (y,\delta(y),\ldots,\delta^j(y))S$. 

Finally, we use induction and (3) again, to obtain $\delta^j(a\delta(y)) \in (\delta(y),\ldots,\delta^{j+1}(y))S$. Putting it all together, we get $z_1 \in (y,\delta(y),\ldots,\delta^n(y))S$, as desired.

For (2), using the fact that $\delta(F(x)) \in (x,\delta(x),p\delta^2(x))S$, which we already proved, we can write $\delta(F(x)) = a\delta(x)+b$, where $b \in (x,p\delta^2(x))S$. Then by (1) we have
$$\delta^n(F(x)) = \delta^{n-1}(a\delta(x) + b) \in (\delta^{n-1}(a\delta(x)),b,\delta(b),\ldots,\delta^{n-1}(b))S.$$ 
Again by the induction hypotheses on (3), we get $\delta^{n-1}(a\delta(x)) \in (\delta(x),\ldots,\delta^{n}(x),p\delta^{n+1}(x))S$. Moreover, for all $1 \leq j \leq n-1$ we have $\delta^j(b) \in (\delta^j(p\delta^2(x)),x,\delta(x),\ldots,\delta^j(x))S$, since $b \in (p\delta^2(x),x)S$. By induction, $\delta^j(p\delta^2(x)) \in (\delta^2(x),\delta^3(x),\ldots,\delta^{j+1}(x),p\delta^{j+2}(x))S$. Therefore 
$$\delta^n(F(x)) \in (x,\delta(x),\ldots,\delta^n(x),p\delta^{n+1}(x))S,$$ 
as claimed.

We finally prove (3). We have
\[
\delta^n(xy) = \delta^{n-1}(x^p\delta(y) + F(y)\delta(x)) = \delta^{n-1}(x^p\delta(y)) + z_2',
\]
where 
$$z_2' \in \left( F(y)\delta(x), \delta(F(y)\delta(x)), \ldots, \delta^{n-1}(F(y)\delta(x) \right) \subseteq \left( F(y),\delta(F(y)),\ldots,\delta^{n-1}(F(y))\right)S,$$ 
by (1) and (3). By (2), for all $1 \leq j \leq n-1$ we have $\delta^j(F(y)) \in (y,\delta(y),\ldots,\delta^j(y),p\delta^{j+1}(y))S$, and thus $z_2' \in (y,\delta(y),\ldots,\delta^{n-1}(y),p\delta^n(y))S$. Finally, by induction $\delta^{n-1}(x^p\delta(y)) = x^{p^n}\delta^n(y) + z_2''$, where $z_2'' \in (\delta(y),\ldots,\delta^{n-1}(y),p\delta^n(y))S$. Setting $z_2=z_2'+z_2''$ concludes the proof.
\end{proof}

Given a ring with a $p$-derivation, we can define mixed differential powers.

\begin{definition}
Let $S$ be an $A$-algebra with a $p$-derivation $\delta$. Given an integer $n$ and a prime ideal $Q$ of $S$, we define the $n$th mixed differential power of $Q$ as
\[
Q\difM{n} = \{f \in S \mid (\delta^a \circ \partial)(f) \in Q \text{ for all } \partial \in D_{S|A}^b \text{ with } a+b \leq n-1\}.
\]
\end{definition}

Mixed differential powers were introduced in \cite{ZNmixed} in order to obtain a Zariski--Nagata type theorem in mixed characteristic for prime ideals that contain the integer $p$.

\begin{theorem}[\protect{\cite[Theorem B]{ZNmixed}}]
Let $p$ be a prime integer, and $S=A[x_1,\ldots,x_m]$, where $A$ is either $\ZZ$ or a discrete valuation ring with uniformizer $p$ and perfect residue field that has a $p$-derivation. If $Q$ is a prime ideal of $S$ that contains $p$, then $Q^{(n)} = Q\difM{n}$.
\end{theorem}

\begin{setting}\label{first setting}
Let $A$ be either $\ZZ$ or a discrete valuation ring with uniformizer $p$ with a $p$-derivation. We will let $S=A[x_1,\ldots,x_m]$ be a graded polynomial ring over $A$, with $\deg(x_i)>0$. Let $\delta$ be a $p$-derivation on $S$, which exists by \cite[Proposition 2.7]{ZNmixed}. Given homogeneous elements $f_1,\ldots,f_t$, let $R=A[f_1,\ldots,f_t]$ be the graded subring of $S$ generated by such elements.
\end{setting}

Setting \ref{first setting} applies, for instance, if $A$ is a complete discrete valuation ring; see \cite[Proposition 2.7]{ZNmixed}.

\begin{definition} 
Let $R \subseteq S$ be an inclusion of graded rings. We say that $R$ is a graded direct summand of $S$ if there exists a degree-preserving $R$-linear map $\beta\!:S \longrightarrow R$ that splits the inclusion of $R$ into $S$. 
\end{definition}

We will use bars over objects to denote residue classes modulo $p$; for instance, $\ov{R}$ will denote $R/(p)$.

\begin{definition} \label{Defn Dq}
Assume Setting \ref{first setting}, and that $\ov{R}$ is a graded direct summand of $\ov{S}$, with graded splitting $\ov{\beta}:\ov{S} \to \ov{R}$. Given an ideal $\a \subseteq R$, and $n \in \NN$, we let 
\[
\a\Dq{n} \colonequals \{x \in R \mid \ov{\beta}(\delta^a(\partial(x)) \ov{S}) \subseteq \ov{\a} \text{ for all } \partial \in D^b_{S|A} \text{ with } a,b \geq 0 \text{ and } a+b < n\}.
\]
\end{definition}

Observe that $a\Dq{0}=R$ and $\a\Dq{1} = \a+(p)$. It follows directly from the definition that $\a\Dq{n+1} \subseteq \a\Dq{n}$ for all $n \geq 1$. 

\begin{remark}
We do not know whether under our assumptions one can always find a $p$-derivation on $S$ that restricts to $R$. If this is the case, then one could also consider the mixed differential powers $\a\difM{n}$ on $R$. See Section \ref{Section comparison} for a comparison between these two types of powers, and related discussions.
\end{remark}

We now show that the differential powers $\a\Dq{n}$ have desirable properties, which resemble those of mixed differential powers as presented in \cite{ZNmixed}.

\begin{proposition} \label{Proposition Dq}
Assume Setting \ref{first setting}, and that $\ov{R}$ is a graded direct summand of $\ov{S}$, with graded splitting $\ov{\beta}:\ov{S} \to \ov{R}$.
Let $\a$ be an ideal of $R$ containing $p$. For all $n \geq 1$, $\a\Dq{n}$ is an ideal. Moreover, if $Q$ is a prime ideal containing $p$, then $Q\Dq{n}$ is $Q$-primary, and it contains~$Q^{(n)}$.
\end{proposition}

\begin{proof}
We start with the first claim, proceeding by induction on $n \geq 1$. Since $\a\Dq{1} = \a$ the base case is trivial. 

Let $x,y \in \a\Dq{n}$. For $0\leq s+t<n$ we consider $\delta^s(\partial(x+y))$, where $\partial \in D^t_{S|\ZZ}$. Since $\partial(x+y)=\partial(x)+\partial(y)$, by Lemma \ref{Lemma Prop delta} (1) we have
\[
\delta^s(\partial(x+y)) \in (\delta^s(\partial(x)),\partial(y),\delta(\partial(y)),\ldots,\delta^s(\partial(y)))S.
\]
Thus
\[
\ov{\beta}(\delta^s(\partial(x+y)) \ov{S}) \subseteq \left(\ov{\beta}(\delta^s(\partial(x))\ov{S}),\ov{\beta}(\delta^j(\partial(x)) \cdot S) \mid 0 \leq j \leq s \right) \subseteq \ov{\a}.
\]
This shows that $x+y \in \a\Dq{n}$ for all $x, y \in \a\Dq{n}$.

Now let $r \in R$ and $x \in \a\Dq{n}$. Again, for $0 \leq s+t < n$ we consider $\delta^s(\partial(rx))$. Recall that $\partial(rx) = r\partial(x) + \partial'(x)$, where $\partial' = [\partial,r] \in D^{t-1}_{S|\ZZ}$. By Lemma \ref{Lemma Prop delta} (1) we have that 
\[
\delta^s(\partial(rx)) \in (\delta^s(r\partial(x)),\partial'(x),\delta(\partial'(x)),\ldots,\delta^s(\partial'(x))S.
\]
By Lemma \ref{Lemma Prop delta} (3) we also have $\delta^s(r\partial(x)) \in (\partial(x),\delta(\partial(x)),\ldots,\delta^{s}(\partial(x))S$. It follows that $\delta^s(\partial(rx)) \in \left(\delta^i(\partial(x)),\delta^j(\partial'(x)) \mid 0 \leq i,j \leq s\right)S$. Thus, $\ov{\beta}(\delta^s(\partial(rx))\ov{S}) \subseteq \ov{\a}$, which shows that $rx \in \a\Dq{n}$. We conclude that $\a\Dq{n}$ is an ideal.

Now let $Q$ be a prime containing $p$. We first show that $Q^n \subseteq Q\Dq{n}$. To this end, it was proven in \cite[Propositions 3.2 and 3.13]{ZNmixed} that $\delta^s(\partial(Q^nS)) \subseteq QS$ for all $\partial \in D^t_{S|\ZZ}$ with $0 \leq s+t<n$. Thus $\ov{\beta}(\delta^s(\partial(x))\ov{S}) \subseteq \ov{Q}$ for all $x \in Q^n$, and we conclude that $Q^n \subseteq Q\Dq{n}$.

Finally, we prove by induction on $n \geq 1$ that $Q\Dq{n}$ is $Q$-primary. The base case is clear: $Q\Dq{1} = Q$. Note that the radical of $Q\Dq{n}$ is $Q$, since $Q^n \subseteq Q\Dq{n} \subseteq Q$. Now let $x,y \in R$ be such that $xy \in Q\Dq{n}$, and $x \notin Q$. Let $s,t$ be nonnegative integers such that $s+t < n$, and let $\partial \in D^t_{S|\ZZ}$. Recall that $x\partial(y) = \partial(xy) + \partial'(y)$, where $\partial' = -[\partial,x] \in D^{t-1}_{S|\ZZ}$. Observe that $xy \in Q\Dq{n} \subseteq Q\Dq{n-1}$ implies, by induction, that $y \in Q\Dq{n-1}$. In particular, $\ov{\beta}(\delta^j(\partial'(y))\ov{S}) \subseteq \ov{Q}$ for all $0 \leq j \leq s$. By Lemma \ref{Lemma Prop delta} (1) we have
\[
\delta^s(x\partial(y)) \in (\delta^s(\partial'(y)), \partial(xy),\delta(\partial(xy)),\ldots,\delta^s(\partial(xy)))S.
\]
We conclude that $\ov{\beta}(\delta^s(x\partial(y))\ov{S}) \subseteq \ov{Q}$. By Lemma \ref{Lemma Prop delta} (3) we have $\delta^s(x\partial(y)) = x^{p^s} \delta^s(\partial(y)) + z$, where $z \in (\partial(y),\delta(\partial(y)),\ldots,\delta^{s-1}(\partial(y)),p)S$. Thus, $\ov{\beta}(z\ov{S}) \subseteq \ov{Q}$. It follows that 
\[
\ov{\beta}(x^{p^s}\delta^s(\partial(y))\ov{S}) = (x^{p^s}) \cdot \ov{\beta}(\delta^s(\partial(y))\ov{S}) \subseteq \ov{Q},
\]
where we use that $x^{p^s} \in R$. Since $x \notin Q$, we conclude that $\ov{\beta}(\delta^s(\partial(y))\ov{S}) \subseteq \ov{Q}$. As $\partial \in D^t_{S|\ZZ}$ and $s,t$ with $s+t < n$ were arbitrary, we conclude that $y \in Q\Dq{n}$.

Given that $Q\Dq{n}$ is $Q$-primary and $Q\Dq{n} \supseteq Q^n$, we conclude that $Q\Dq{n}$ contains $Q^{(n)}$.
\end{proof}

\section{A uniform Chevalley Theorem} \label{Section main}

\begin{notation}\label{notation}
Assume Setting \ref{first setting}. Let $\q = (f_1,\ldots,f_t)R$, $\eta=(x_1,\ldots,x_m)S$, $\m = \q+pR$ and $\n = \eta + (p)$.
\end{notation}

Observe that $\eta^{nD} = \eta^{(nD)} =  \eta^{\langle nD \rangle_A}$ and $\n^{nD} = \n^{(nD)} = \n^{\langle nD \rangle_{\rm mix}}$, where in both cases the first equality is because $\eta$ and $\n$ are generated by a regular sequence, and the second equality follows from \cite[Theorem A and Theorem B]{ZNmixed}.

We are now ready to state the main results concerning a uniform Chevalley theorem for symbolic powers in mixed characteristic. We start with the case of homogeneous primes that do not contain prime integers, for which the proof is similar to that of \cite[Theorem 3.27]{SurveySP}, and does not require the use of the new differential powers introduced in this article.

\begin{theorem} \label{THM no p}
Assume Setting \ref{first setting}, and that $R$ is a graded direct summand of $S$. Using Notation \ref{notation}, let $D=\max\{\deg(f_1),\ldots,\deg(f_t)\}$. Then $Q^{(nD)} \subseteq \q^n$ for all $n \geq 1$ and all prime ideals $Q \subseteq \q$. 
\end{theorem}

\begin{proof}
Let $\beta\!:S \to R$ be a splitting of the inclusion. First we show that $Q^{(nD)} \subseteq R \cap \eta^{nD}$. Consider an element $f \in R$ that does not belong to $\eta^{nD}$. By \cite[Theorem A]{ZNmixed}, there exists a differential operator $\partial \in D^{n-1}_{S|A}$ such that $\partial(f) \notin \eta$, that is, $\partial(f) = a + g$ where $a \in A \smallsetminus \{0\}$ and $g \in \eta$. Since $\beta$ is graded and $A$-linear, we have that $(\beta \circ \partial)_{|R} \in D^{n-1}_{R|A}$ (see \cite[Lemma 3.1]{AMHNB}), and $\beta(\partial(f)) =a+\beta(g) \notin \q$, since $\beta(g) \in \beta(\eta) \subseteq \q$. Since $Q \subseteq \q$, it follows that $\beta(\partial(f)) \notin Q$, so $f \notin Q^{\langle nD \rangle_A}$, and therefore $f \notin Q^{(nD)}$ by \cite[Proposition 3.2 (5)]{ZNmixed}.

Now we show that $R \cap \eta^{nD} \subseteq \q^n$, so that $Q^{(nD)} \subseteq \q^n$ will follow. To show the containment, we let $g \in R \cap \eta^{nD}$ be an element, which we may assume to be homogeneous with respect to the grading of $S$. If we let $d=\deg(g)$, then note that $d \geq nD$. Since $g \in R$, we may write $g=\sum_{\alpha} c_\alpha f_1^{\alpha_1} \cdots f_t^{\alpha_t}$ for some $c_\alpha \in A$. Let $\alpha = \min\{\alpha_1 + \cdots + \alpha_t \mid c_\alpha \ne 0\}$, so that $d= \sum_{i=1}^t \alpha_i\deg(f_i) \leq (\alpha_1 + \cdots + \alpha_t)D$, from which we get $\alpha \geq d/D \geq n$. Since $g \in \q^\alpha$, the proof is complete.
\end{proof}

For primes that do contain a prime integer $p$, on the other hand, the use of the new type of mixed differential powers introduced in this article (Definition \ref{Defn Dq}) will be crucial. 

\begin{theorem} \label{main THM}
Assume Setting \ref{first setting}, and that $\ov{R}$ is a graded direct summand of $\ov{S}$. Using Notation \ref{notation}, let $D=\max\{\deg(f_1),\ldots,\deg(f_t)\}$. Then $Q^{(nD)} \subseteq \m^n$ for all $n \geq 1$ and all prime ideals $Q \subseteq \m$ that contain $p$.
\end{theorem}

\begin{proof}
Let $\ov{\beta}\!:\ov{S} \to \ov{R}$ be a graded splitting of the inclusion $\ov{R} \subseteq \ov{S}$. Let $f \in R$ be an element that does not belong to $\n^{nD}$. By \cite[Theorem B]{ZNmixed} there exist integers $s,t$ with $s+t \leq nD-1$ and a differential operator $\partial \in D^t_{S|\ZZ}$ such that $\delta^s \circ \partial(f) \notin \n$. This means that we can write $\delta^s \circ \partial(f) = a+g$, where $g \in \eta$ and $a \in A$ is such that $(a,p) = A$. It follows that $\ov{\beta}(\delta^s(\partial(f))) = \ov{a}+\ov{\beta}(\ov{g}) \notin \ov{\m}$, since $\ov{\beta}(\ov{g}) \in \ov{\beta}(\ov{\eta}) \subseteq \ov{\q} \subseteq \ov{\m}$ and $\ov{a} \notin \ov{\m}$. Since $Q \subseteq \m$, we conclude that $\ov{\beta}(\delta^s(\partial(f)))  \notin \ov{Q}$. It follows that $f \notin Q\Dq{nD}$ and, by Proposition \ref{Proposition Dq}, we have that $f \notin Q^{(nD)}$. This shows that $Q^{(nD)} \subseteq R \cap \n^{Dn}$.

We now claim that $R \cap \n^{nD} \subseteq \m^{n}$, which will complete the proof. In order to prove this claim, we let $g \in \n^{nD} \cap R$ be homogeneous with respect to the grading in $S$. Since $g \in \n^{nD}$, there exists an integer $0 \leq d \leq nD$ such that $g = p^{nD-d} \widetilde{g}$ for some $\widetilde{g} \in S \smallsetminus pS$. If $d=0$ there is nothing to show, since clearly $g \in \m^n$ in this case. In what follows, let us assume otherwise. Since $\ov{R}$ is a direct summand of $\ov{S}$, and $p$ is a regular element on $S$, we have that $(p^{nD-d})S \cap R = (p^{nD-d})R$, and therefore we may assume that $\widetilde{g} \in R$. We can then write
$$\widetilde{g} = \sum_{\underline{\alpha}} c_{\underline{\alpha}} f_1^{\alpha_1} \cdots f_t^{\alpha_t}$$
for some $c_{\underline{\alpha}} \in A$. Let $\alpha = \min\{\alpha_1 + \cdots + \alpha_t \mid \underline{\alpha} = (\alpha_1,\ldots,\alpha_t), c_{\underline{\alpha}} \ne 0\}$, so that $\widetilde{g} \in \q^\alpha$. Note that $\widetilde{g} \in (\n^{nD}:_Sp^{nD-d}) \subseteq \n^d$, and that $\widetilde{g}$ must actually have a monomial in $\eta^d$ with nonzero coefficient, otherwise it would belong to $pS$. Since $\widetilde{g}$ is homogeneous, it follows that $\deg(\widetilde{g}) \geq d$. Moreover, $d=\alpha_1 \deg(f_1) + \cdots + \alpha_t \deg(f_t) \leq (\alpha_1+\cdots+\alpha_t) D$ for every $\underline{\alpha} = (\alpha_1,\ldots,\alpha_t)$ such that $c_{\underline{\alpha}} \ne 0$, and therefore $\alpha \geq d/D$. We conclude that $g  = p^{nD-d}\widetilde{g} \in (p)^{nD-d} \q^\alpha \subseteq  \m^{n}$, where the last containment follows from the fact that $nD-d+\alpha \geq nD-d+d/D \geq n$, using that $nD \geq d$.
\end{proof} 

\begin{example} \label{Example sharp}
Let $A=\ZZ_p$ be the ring of $p$-adic integers, and $S=A[s,t]$. Let $R$ be the ring of invariants under the action of $\ZZ/(2) \cong \{-1,1\}$ that sends $s \mapsto \pm s$ and $t \mapsto \pm t$. It can be shown that $R = A[s^2,st,t^2] \cong A[x,y,z]/(xz-y^2)$, so that $D=2$ in this case. If we let $Q=(x,y)$, then $Q^{(2n)} = (x^n)$ for all $n \geq 1$, therefore the containments of Theorem \ref{THM no p} and \ref{main THM} are sharp for this ring.
\end{example}

\begin{remark} \label{Remark direct summand mod p}
The assumption of Theorem \ref{main THM} that $\ov{R}$ is a (graded) direct summand of $\ov{S}$ is potentially weaker than the condition of Theorem \ref{THM no p} that $R$ is a (graded) direct summand of $S$. For instance, Jeffries and Singh prove that certain determinantal rings are not direct summands of any regular local ring in mixed characteristic \cite[Corollary 1.5]{JS}. The validity of the corresponding statement in characteristic $p>0$, however, is still unknown.
\end{remark}

\begin{corollary} \label{main Coroll} 
Let $A$ be a discrete valuation ring with uniformizer $p$ and perfect residue field, and assume that $A$ has a $p$-derivation. If $R=A[f_1,\ldots,f_t]$ is a graded direct summand of $S=A[x_1,\ldots,x_m]$, $\m = (p, f_1, \ldots, f_t)$, and $D=\max\{\deg(f_1),\ldots,\deg(f_t)\}$, then $Q^{(nD)} \subseteq \m^n$ for all $n\geq 1$ and all homogeneous prime ideals $Q \subseteq R$.
\end{corollary}

\begin{proof}
Let $Q \subseteq R$ be a homogeneous prime ideal. If $Q \cap A = (0)$, then $Q$ must be contained in the positive degree part of $R$, that is, we must have $Q \subseteq \q$. By Theorem \ref{THM no p} we then have that $Q^{(nD)} \subseteq \q^n \subseteq \m^n$ for all $n \geq 1$. On the other hand, if $Q \cap A = (p)$, then $Q^{(nD)} \subseteq \m^n$ follows from Theorem \ref{main THM}, as our assumptions guarantee that $\ov{R}$ is a graded direct summand of $\ov{S}$.
\end{proof}

As a consequence of our results, and of several improvements of the famous Noether bound on the degree of generators of rings of invariants, we obtain the following.

\begin{theorem} \label{THM Noether bound} Let $p$ be a prime integer, and $A$ be a discrete valuation ring with uniformizer $p$ and perfect residue field, and assume that $A$ has a $p$-derivation. Let $G$ be a finite group of order $D$ coprime with $p$ acting on $S=A[x_1,\ldots,x_n]$, where $\deg(x_i)=1$ for all $i$. Let $R=S^{G}$ be the ring of invariants, and $\m = R_{>0} + pR$. Then $Q^{(nD)} \subseteq \m^n$ for every homogeneous prime ideal $Q \subseteq R$ and every integer $n$.
\end{theorem}
\begin{proof}
This follows from \cite[Theorem 3.1]{Fleischmann} (see also \cite{Fogarty}) and Corollary \ref{main Coroll}.
\end{proof}

\section{Comparison between differential and symbolic powers and Examples} \label{Section comparison}

We continue with the same notation introduced in the previous sections, focusing only on the case in which $A$ is a discrete valuation ring with uniformizer $p$ and perfect residue field, and that $A$ has a $p$-derivation. One key feature of this article is the introduction of a new differential power in mixed characteristic that applies to ideals of $R$ whenever the inclusion $\ov{R} \subseteq \ov{S}$ splits. In this last section we compare it with the notion already available in the literature, and study its relations with symbolic powers in more details.

Let $Q$ be a prime ideal of $R$ that contains $p$. We want to compare the differential power $Q\Dq{n}$ with the mixed differential power $Q\difM{n}$. However, in order to even define the latter, we need to assume that $R$ has a $p$-derivation. It is therefore natural to put, as an extra assumption, that there is a $p$-derivation on $S$ that restricts to a $p$-derivation on $R$.

\begin{setting}\label{second setting}
	Let $A$ be a discrete valuation ring with uniformizer $p$. Let $S=A[x_1,\ldots,x_m]$ be a graded polynomial ring over $A$, with $\deg(x_i)>0$. Assume that $S$ has a $p$-derivation $\delta_S:S \to S$ that restricts to a $p$-derivation $\delta_R:R \to R$. Given homogeneous elements $f_1,\ldots,f_t$, let $R=A[f_1,\ldots,f_t]$ be the graded subring of $S$ generated by such elements.
\end{setting}

The following condition will be of interest below:

\begin{definition} Let $A$ be a ring. An inclusion $R\subseteq S$ of $A$-algebras is \emph{(order) differentially extensible} if for every $\partial\in D^n_{R|A}$, there is some $\widetilde{\partial}\in D_{S|A}$ such that $\widetilde{\partial}_{|_R}=\partial$.
	\end{definition}

Several classical invariant subrings of polynomial rings are differentially extensible; for instance, Veronese subrings of $A[x_1,\ldots,x_m]$, with $m \geq 2$ are differentially extensible \cite[Proposition 6.4]{DiffSig}. We note that while \cite[Proposition 6.4]{DiffSig} is written over fields, the same argument works more generally.

If we assume that $R \subseteq S$ is differentially extensible, then given a prime ideal $Q$ containing $p$ we always have containments
\[
Q^n \subseteq Q^{(n)} \subseteq Q\Dq{n} \subseteq Q\difM{n}.
\]
Thanks to Proposition \ref{Proposition Dq}, the only containment that is not already clear is $Q\Dq{n} \subseteq Q\difM{n}$.

\begin{lemma} \label{Lemma containment new and mixed diff}
Assume Setting \ref{second setting}. Let $R=A[f_1,\ldots,f_t]$ be such that $\ov{R}$ is a graded direct summand of $\ov{S}$. Assume that the inclusion $R \subseteq S$ is differentially extensible. For all $n \geq 1$ and all prime ideals $Q$ containing $p$, we have
\[
Q\Dq{n} \subseteq Q\difM{n}.
\]
\end{lemma}

\begin{proof}
Let $\delta_R\!:R \to R$ be the restriction of $\delta_S$ to $R$. Let $x \in R$ be such that $x \notin Q\difM{n}$. Then, there exist integers $a,b$ with $a+b<n$ and $\partial \in D_{R|A}^b$ such that $\delta_R^a(\partial(x)) \notin Q$. Since we are assuming that $R \subseteq S$ is differentially extensible, there exists $\partial' \in D_{S|A}^b$ such that $\partial'_{|R} = \partial$. In particular, $\partial(x) = \partial'(x) \in R$. It follows that $\delta_S(\partial'(x)) = \delta_R(\partial(x)) \notin Q$, and therefore $x \notin Q\Dq{n}$.
\end{proof}

On the other hand, if the inclusion $R \subseteq S$ is not differentially extensible, then it is easy to find examples where the containment of Lemma \ref{Lemma containment new and mixed diff} does not hold.

\begin{example} \label{Example non-differentially-extensible}
Consider the split inclusion $R=\ZZ_p[x^2] \subseteq S=\ZZ_p[x]$. This inclusion is not differentially extensible, as differentiation by the variable $x^2$ in $R$ does not extend to a differential operator on $S$ of the same order. In this setting, for $\m=(p,x^2)\subseteq R$, we claim that $x^2\in \m\Dq{2} \smallsetminus \m\difM{2}$. To see that $x^2\notin \m\difM{2}$, we note that $\m\difM{2}=\m^2$ by \cite[Theorem B]{ZNmixed}, and because $\m$ is a maximal ideal. On the other hand, $(\delta_S\circ  D^0_{S|\ZZ_p})(x^2)$ and $D^1_{S|\ZZ_p}(x^2)$ are both contained in $xS$. Therefore, after applying the splitting $\ov{\beta}: \ov{S} \to \ov{R}$, their images are contained in $x^2 \ov{R}$. It follows that $x^2\in \m\Dq{2}$, as required.
\end{example}

\begin{remark}\label{remark p derivation p not in second symbolic power}
Suppose that $R$ is a ring with a $p$-derivation $\delta$, and $Q$ is a prime in $R$ containing $p$. Then $\delta(p) \notin Q$, so $p \notin Q\difM{2}$, which in particular implies $p \notin Q^{(2)}$ by \cite[Proposition 3.19 (4)]{ZNmixed}.
\end{remark}

\begin{theorem} \label{Proposition equality regular prime}
Assume Setting \ref{second setting}. Suppose that $\ov{R}$ is a direct summand of $\ov{S}$, and assume that the inclusion $R \subseteq S$ is differentially extensible.  If $Q \in \Spec(R)$ is a prime that contains $p$ and such that $R_Q$ is regular, then $Q^{(n)} = Q\Dq{n} = Q\difM{n}$.
\end{theorem}

\begin{proof}
Since all the ideals involved in the statement of this proposition are $Q$-primary, it suffices to show the equalities after localizing at $Q$. By Remark \ref{remark p derivation p not in second symbolic power}, $\ov{R_Q}$ is regular, and therefore $R_Q$ is essentially smooth over $A$ since it is flat with geometrically regular fibers. By \cite[Lemma 3.20 and Proposition 3.22]{ZNmixed} we conclude that 
\[
(Q\difM{n})_Q = (Q_Q)\difM{n} = (Q_Q)^{(n)} = (Q^{(n)})_Q.
\]
This shows $Q^{(n)} = Q\difM{n}$. Since $Q^{(n)} \subseteq Q\Dq{n} \subseteq Q\difM{n}$, the proof is complete.
\end{proof}

In equal characteristic, it is shown under mild assumptions that the equality between differential powers and symbolic powers of $Q$ implies that $Q$ is a regular prime (see \cite[Theorem 10.2]{DiffSig}). We do not know whether the same also holds in mixed characteristic with mixed differential powers in place of differential powers.

Example \ref{Example non-differentially-extensible} shows that there are inclusion of algebras that are not differentially extensible for which $Q\Dq{n}$ and $Q\difM{n}$ do not coincide. It is therefore natural to ask the following question.

\begin{question} \label{Question equality} Under the assumptions of Lemma \ref{Lemma containment new and mixed diff}, is $Q\Dq{n}=Q\difM{n}$ for all $n \geq 1$?
\end{question}

We point out that, if Question \ref{Question equality} has a positive answer, then $Q\Dq{n}$ would be a natural extension of $Q\difM{n}$ to direct summands that, in principle, might not have a $p$-derivation of their own. On the other hand, if the answer was negative, $Q\Dq{n}$ would still be a better approximation of our main object of study, the symbolic powers $Q^{(n)}$, than $Q\difM{n}$ is.

These two notions of mixed differential powers can coincide even if they are not the symbolic powers.

\begin{example}
	Consider $R = \ZZ_p[x^2,xy,y^2] \subseteq S = \ZZ_p[x,y]$ and let $\m = (p, x^2,xy,y^2)$ be the homogeneous maximal ideal in $R$. Fix the graded splitting $\ov\beta\!: \ov{S} \to \ov{R}$ given by projecting onto the even degree components, and let $\delta$ be a $p$-derivation on $S$. Note that $R \subseteq S$ is differentially extensible.
	
	Since $\m$ is a maximal ideal, $\m^{(2)} = \m^2$. However, we claim that $\m^{(2)} \subsetneq \m\Dq{2} = \m\difM{2}$. In $S$, $D^1_{S | \ZZ_p}(x^2) \subseteq (x)$, $D^1_{S | \ZZ_p}(y^2) \subseteq (y)$, and $D^1_{S | \ZZ_p}(xy) \subseteq (x,y)$. Therefore, 
	$$\ov{\beta}((D^1_{S | \ZZ_p}(x^2),D^1_{S | \ZZ_p}(xy),D^1_{S | \ZZ_p}(y^2))\ov{S}) \subseteq \ov{\m}.$$ 
	Moreover, if $\delta_S$ is any $p$-derivation on $S$, we have that $\delta_S(x^2) \in (x,p)$, $\delta_S(y^2) \in (y,p)$, and $\delta_S(xy) \in (x,y,p)$, so 
	$$\ov{\beta}((\delta_S(x^2),\delta_S(xy),\delta_S(y^2))\overline{S}) \subseteq \ov{\m}.$$ 
	This shows that $x^2, xy, y^2 \in \m\Dq{2}$. Moreover, $\ov{\delta_S(p)} = \ov{1}$, so $p \notin \m\Dq{2}$. On the other hand, $\ov{\delta_S(p^2)} = 0$, so $p^2 \in \m\Dq{2}$. We conclude that $\m\Dq{2} = (p^2,x^2,xy,y^2)$, which implies that $\m\difM{2} \supseteq (p^2,x^2,xy,y^2)$. Since $\delta_R(p) \notin \m$ for any $p$-derivation on $R$, we have that $p \notin \m\difM{2}$. Finally, since $\m\difM{2} \subseteq \m$, we conclude that $\m\difM{2} = (p^2,x^2,xy,y^2)$.
\end{example}

We end the article with two examples that show how the main results can be applied.

\begin{example} 
Let $R$ be the $D$th Veronese subring of $S=A[x_1,\ldots,x_m]$, where $\deg(x_i)=1$ for all $i$. Let $\m$ be the maximal ideal of $R$ generated by $p$ and the monomials of degree $D$ in the variables of $S$. It follows from Corollary \ref{main Coroll} that $Q^{(nD)} \subseteq \m^n$ for all $n \geq 1$ and all homogeneous primes $Q \subseteq R$. In fact, in our assumptions $R$ is a graded direct sum of $S$, and therefore we can apply Corollary \ref{main Coroll}. Note that, when $p \equiv 1 \bmod D$, then $R$ can also be seen as the ring of invariants under the action on $S$ of a group of order $D$, and the containment  also follows from Theorem \ref{THM Noether bound} in this case.
\end{example}

\begin{example} 
Let $X=(x_{ij})$ be an $m\times n$ matrix of indeterminates, and $R=A[X]/I_2(X)$ be the quotient of $A[X]$ by the ideal generated by the $2 \times 2$ minors of $X$. Let $\m=(X,p)R$. We have that $R\cong A[y_i \cdot z_j \mid i =1,\ldots,m, j=1,\ldots,n]$ is a graded direct summand of $S=A[y_i,z_j \mid i=1,\ldots,m,j=1,\ldots,n]$, where $\deg(y_i) = \deg(z_j) = 1$ for all $i,j$, and therefore by Corollary \ref{main Coroll} we have that $Q^{(2n)} \subseteq \m^n$ for all $n \geq 1$ and all homogeneous primes $Q \subseteq R$.
\end{example}

\section*{Acknowledgements}

We thank Linquan Ma for pointing out that our proofs of Theorems \ref{THM no p} and \ref{main THM} were yielding stronger results than the ones originally stated. We thank the anonymous referee. The first author was partially supported by the PRIN 2020 project 2020355B8Y ``Squarefree Gr{\"o}bner degenerations, special varieties and related topics''. The second author was partially supported by NSF Grant DMS~\#2140355. The third author was partially supported by NSF CAREER Grant DMS~\#2044833.

\bibliographystyle{alpha}
\bibliography{References}

\end{document}